\documentclass[12pt]{amsart}
\usepackage{amsmath}
\usepackage{amsthm}
\usepackage{amsfonts}
\usepackage{amssymb}
\newtheorem{Theorem}{Theorem}[section]
\newtheorem{Proposition}[Theorem]{Proposition}
\newtheorem{Lemma}[Theorem]{Lemma}
\newtheorem{Corollary}[Theorem]{Corollary}
\theoremstyle{definition}
\newtheorem{Definition}{Definition}[section]
\theoremstyle{definition}
\newtheorem{Hypothesis}{Hypothesis}[section]
\theoremstyle{remark}

\numberwithin{equation}{section}

\newcommand{\R}{{\mathbb R}}
\newcommand{\C}{{\mathbb C}}

\newcommand{\SL}{{\textrm{\rm SL}}}
\newcommand{\PSL}{{\textrm{\rm PSL}}}

\newcommand{\tr}{{\textrm{\rm tr}\:}}

\renewcommand{\Im}{{\textrm{\rm Im}\:}}

\begin{document}

\title[Reflectionless operators]{Reflectionless Dirac operators and canonical systems}

\author{Christian Remling}

\address{Department of Mathematics\\
601 Elm Ave.\\
University of Oklahoma\\
Norman, OK 73019}
\email{christian.remling@ou.edu}
\urladdr{www.math.ou.edu/$\sim$cremling}

\author{Jie Zeng}

\address{Department of Mathematics\\
3900 University Blvd.\\
University of Texas at Tyler\\
Tyler, TX 75799}

\email{jzeng@uttyler.edu}

\date{October 25, 2024}



\begin{abstract}
We study canonical systems that are reflectionless on an open set. In this situation, the two half line $m$ functions are holomorphic continuations
of each other and may thus be combined into a single holomorphic function. This idea was explored in \cite{HMcBR}, and we continue these investigations here.
We focus on Dirac operators and especially their interplay with canonical systems, and we provide a more general and abstract framework.
\end{abstract}
\maketitle
\section{Introduction}
We investigate reflectionless Dirac operators and canonical systems in this paper.
A \textit{canonical system }is a differential equation of the form
\begin{equation}
\label{can}
Ju'(x) = -zH(x)u(x) , \quad J=\begin{pmatrix} 0 & -1 \\ 1 & 0 \end{pmatrix} ,
\end{equation}
with a locally integrable coefficient function $H(x)\in\R^{2\times 2}$, $H(x)\ge 0$, $\tr H(x)=1$.
This last condition is a normalization. Something of this sort is needed if one wants to have a bijection
between canonical systems and spectral data, but the fine details are largely arbitrary. The trace normalization is a common choice,
but in the context of our discussion of Dirac operators below, it will actually be more convenient to normalize the determinant:
if $\det H(x)>0$ for all $x\in\R$, then there is a unique version of $H$, with the same $m$ functions, as defined in \eqref{defm} below, that satisfies $\det H(x)=1$.
The general issue is discussed in more detail in \cite[Section 1.3]{Rembook}.

Canonical systems define self-adjoint relations and operators on the Hilbert spaces
\[
L^2_H(I) = \left\{ f: I\to\C^2 : \int_I f^*(x)H(x)f(x)\, dx < \infty \right\} .
\]
They are of fundamental importance in spectral theory because they may be used
to realize arbitrary spectral data \cite[Theorem 5.1]{Rembook}. We will use the notation $\mathcal C (I)$ for the collection
of canonical systems $H(x)$, $x\in I$. In line with what we just discussed, if two (differently normalized)
coefficient functions $H$ share the same $m$ functions, to be introduced below, in \eqref{defm}, we view them as representing the same element of $\mathcal C$.

\textit{Dirac equations }are, of course, more classical. We will work with the following general form:
\begin{equation}
\label{Dirac}
Jy'(x) + W(x)y(x) = -zy(x) .
\end{equation}
Here, we assume that $W(x)\in\R^{2\times 2}$, $W(x)=W^t(x)$, $W\in L^1_{\textrm{loc}}(I)$. We call such (matrix) functions
\textit{Dirac potentials; }as usual, we identify potentials that agree almost everywhere. It will again be convenient to have the
short-hand notation $\mathcal W (I)$ available for the collection of Dirac potentials.
The Dirac equation \eqref{Dirac} will generate self-adjoint operators on $L^2(I;\C^2)$.

Since, as we just mentioned, any spectral data can be realized by a unique canonical system, it must in particular be possible to do this
for those of the Dirac equations \eqref{Dirac}. We will discuss the procedure in detail in Section 2. So, in this sense, Dirac equations can be thought
of as special canonical systems.

A canonical system or Dirac equation on $x\in I=\mathbb R$ is called \textit{reflectionless }on a Borel set $A\subseteq\mathbb R$ if
\begin{equation}
\label{refless}
m_+(t)= -\overline{m_-(t)}
\end{equation}
for (Lebesgue) almost every $t\in A$. Here, $m_{\pm}$ are the \textit{Titchmarsh-Weyl }$m$ \textit{functions }of the half line problems on
$x\in [0,\infty)$ and $x\in (-\infty, 0]$, respectively. These functions are key tools in the spectral analysis of \eqref{can}, \eqref{Dirac}; please see
\cite[Chapter 3]{Rembook} for a detailed treatment. They may be defined as
\begin{equation}
\label{defm}
m_{\pm}(z) = \pm f_{\pm}(0,z) ,
\end{equation}
with $z\in\mathbb C^+=\{ z\in\mathbb C: \Im z >0 \}$ and $f_+=u$ denoting the
(unique, up to a factor) solution $f_+\in L_H^2(0,\infty)$ of \eqref{can}, and $f_-$ similarly denotes the solution that is square integrable on the left half line.
We also use the convenient convention of identifying a vector $f=(f_1,f_2)^t\in\mathbb C^2$, $f\not= 0$, with the point $f_1/f_2\in\mathbb C_{\infty}$
on the Riemann sphere. In particular, $m_{\pm}(z)\in\mathbb C_{\infty}$, and in fact the $m$ functions are generalized \textit{Herglotz functions, }that is,
they map the upper half plane $\mathbb C^+$ holomorphically back to $\overline{\mathbb C^+}=\C^+\cup\R\cup\{\infty \}$.

If specifically $H(x)\equiv P_{\alpha}$ on a half line,
let's say on $x\ge 0$, with $P_{\alpha}$ denoting the projection onto $e_{\alpha}=(\cos\alpha, \sin\alpha)^t$, then $m_+(z)\equiv -\tan\alpha$.
If $H$ is not of this special type, then the $m$ functions are genuine Herglotz functions, that is, they map $\C^+$ back to itself holomorphically.
Herglotz functions have boundary values $m(t)=\lim_{y\to 0+} m(t+iy)$ at almost all $t\in\mathbb R$, and we are referring to these in \eqref{refless}.

If $H(x)\equiv P_{\alpha}$ on $x\in\R$, then $m_{\pm}(z)\equiv \mp\tan\alpha$, and $H$ is reflectionless on $\R$ according to our definition \eqref{refless}.
It will be convenient to have the short-hand notation
\begin{equation}
\label{1.93}
\mathcal Z = \{ H\equiv P_{\alpha} : 0\le\alpha<\pi \}
\end{equation}
available to refer to the collection of these trivial canonical systems. Observe also that if $H$ is reflectionless on any positive measure set and $H\notin\mathcal Z$,
then both half line $m$ functions are genuine Herglotz functions.

Almost literally the same definitions may be used for Dirac equations. We again define the half line $m$ functions by \eqref{defm},
with $f_{\pm}=y$ now denoting square integrable solutions of \eqref{Dirac}. These $m$ functions are genuine Herglotz functions;
the degenerate functions $m\equiv a\in\R_{\infty}$ do not occur for Dirac operators.

Reflectionless operators are important because they can be thought of as the basic building blocks of arbitrary operators with
some absolutely continuous spectrum; compare \cite{RemAnn}, \cite[Chapter 7]{Rembook}. Let's introduce
\[
\mathcal R (A) =\{ H\in\mathcal C(\R): H \textrm{ is reflectionless on }A\} .
\]
In this paper, we are interested in operators that are reflectionless on an \textit{open }set $A=U\subseteq\R$. In this case, the half line $m$ functions
are holomorphic continuations of each other through $U$. More precisely, we have the following result.
\begin{Lemma}
\label{L1.1}
Let $U\subseteq\R$ be a non-empty open set, and assume that $H\in\mathcal R (U)$. Then the function
\[
M(z) = \begin{cases} m_+(z) & z\in\C^+ \\ -\overline{m_-(\overline{z})} & z\in\C^- \end{cases}
\]
has a holomorphic continuation to $\Omega\equiv\C^+\cup U\cup\C^-$.
\end{Lemma}
Here, $\C^-=\{z: \Im z<0\}$ of course denotes the lower half plane.

Lemma \ref{L1.1} certainly looks plausible and is well known.
However, there is slightly more to it than meets the eye since \eqref{refless} is only imposed almost everywhere, and we do need to make sure
that the exceptional null set is empty when $U$ is open. The (uncomplicated) detailed proof of Lemma \ref{L1.1} may be found in \cite{HMcBR},
where the result is formulated as Lemma 2.1.

Lemma \ref{L1.1} has a converse of sorts, which is in fact obvious.
Let's spell this out in more detail. Let $H\in\mathcal C[0,\infty)$ and suppose that $m_+(\cdot ;H):\C^+\to\overline{\C^+}$ has a holomorphic
continuation $M:\Omega\to\overline{\C^+}$. Then $H\in\C[0,\infty)$ is the right half line of a unique $H\in\mathcal R(U)$.

A weaker version of this condition could ask for the existence of such a holomorphic continuation of $m_+$, still mapping to $\overline{\C^+}$, to \textit{some }larger domain,
not necessarily all of $\Omega$. We will later focus on $U=\R\setminus [-1,1]$, and then $\Omega$ has a puncture
at $\infty$, which we can plug, and it is exactly the neighborhood of this point that is of greatest interest, so we specialize to this situation right away.
\begin{Hypothesis}
\label{H1.1}
The half line $m$ function $m_+$ has a holomorphic continuation to a neighborhood of $z=\infty$ that still maps to $\overline{\C^+}$.
More explicitly, there are a disk $D_r=\{|z|<r\}$ and a holomorphic function $M: D_r\to\overline{\C^+}$ such that $m_+(z)=M(-1/z)$
for all $z\in\C^+$, $|z|>1/r$.
\end{Hypothesis}
\begin{Lemma}
\label{L1.2}
Suppose that $H\in\mathcal R(U)$ for some open set $U\subseteq\R$ with bounded complement. Then $m_+(z;H)$ satisfies Hypothesis \ref{H1.1}.
\end{Lemma}
\begin{proof}
Lemma \ref{L1.1} provides a holomorphic continuation $M$ of $m_+$ to a domain $\Omega$ that contains a punctured
neighborhood of $\infty$. Since $M$ maps to $\overline{\C^+}$, the singularity is removable.
\end{proof}
We originally got interested in Hypothesis \ref{H1.1} because this condition helped streamline some parts of our presentation.
However, these canonical systems also seem of independent interest. We will prove the following here, in Sections 5, 6.
\begin{Theorem}
\label{T1.2}
Let $H\in\mathcal C[0,\infty)$, and assume Hypothesis \ref{H1.1}. Then either $H\equiv P_{\alpha}$, or $\det H(x)=1$ for all $x\ge 0$, for a suitably
normalized version of $H$.

In the second case, $H(x)$ has the following additional properties: $H(x)$ is real analytic. More precisely, $H$ has a holomorphic
continuation $H(z)$ to a strip
\[
\{ z\in\C: |z-x|<h \textrm{\rm{ for some }}x>0\} .
\]
Moreover, there is a constant $C>0$ such that
\[
\|H^{-1/2}(x)H'(x)H^{-1/2}(x)\| \le C
\]
for all $x\ge 0$.
\end{Theorem}

Now let's return to the situation from Lemma \ref{L1.1}.
In the simplest case, when $U=I$ is a single interval, $\Omega$ is simply connected and conformally equivalent to $\C^+$, if we also exclude
the (rather trivial) case $I=\R$ here, which would give $\Omega=\C$.
Fixing such a map $\varphi: \C^+\to\Omega$, we may then introduce the new Herglotz function $F(\lambda)=M(\varphi(\lambda))$. Moreover,
convenient explicit maps $\varphi$ are available in these simple cases.

This will set up a one-to-one correspondence between $\mathcal R (U)$ and Herglotz functions $F$ (= Theorem \ref{T3.1} below), and it will allow
us to draw interesting conclusions about individual systems $H\in\mathcal R(U)$.

Since both \cite{HMcBR} and the present work start from exactly the same place, namely Lemma \ref{L1.1}, let us also briefly discuss
how we go beyond \cite{HMcBR} here. First of all, we give a detailed discussion of the Dirac operators from $\mathcal R (U)$.
This case was not discussed in \cite{HMcBR}, which instead dealt with Jacobi and Schr{\"o}dinger operators exclusively. This is curious as arguably the method works
best for Dirac operators. In any event, we fill this gap here. We also use canonical systems as our general framework, thus expanding the scope of the
method, and actually simplifying it in the process. This broad and general viewpoint will already pay off in the next (preparatory) section, when we discuss basic material
on Dirac operators. Much of what we do there is of course also discussed in standard references on the subject such as
\cite{LevSar}, but we do hope that our more conceptual approach, which emphasizes the connection to canonical systems and
abstract notions, may have something to offer. Compare also \cite{BLY}.

In Section 3, we then give the anticipated description of $\mathcal R(U)$ and clarify where the Dirac operators are in this space; see Theorems \ref{T3.1}, \ref{T3.2} and
perhaps also Theorem \ref{T5.1}. The answer to this second question (= Theorem \ref{T3.2}) can be restated in abstract style: the Dirac operators in $\mathcal R(U)$
form a compact convex set in a natural way. See Theorem \ref{T3.3} please. A tantalizing aspect of this reformulation is that it would allow us to recover
the characterization of Theorem \ref{T3.2} of the Dirac equations within the larger space $\mathcal R(U)$ without any recourse to the highly technical machinery
of inverse spectral theory if we had an independent proof of it (but of course we don't). We'll explain this in more detail at the end of Section 3.

We can easily identify the extreme points of this compact convex set, and they have a convenient
spectral theoretic description also. See Theorem \ref{T3.3} again.

Section 4 then analyzes Dirac operators from $\mathcal R(U)$ in some detail, following ideas from \cite{HMcBR,Mar}.
As in the Schr{\"o}dinger case that was analyzed in these works, it turns out that $W(x)$ is real analytic and obeys sharp bounds. See Theorem \ref{T4.1},
and compare also the discussion of Theorem \ref{T1.2} above.

In Section 5, we prove a version of Theorem \ref{T1.2} for Dirac operators.
In the final Section 6, we return to the broader perspective of canonical systems. In particular, we give versions of the results of Section 4 that apply in
this generality, and we can then at last also prove Theorem \ref{T1.2} as stated.

\textit{Acknowledgment. }We thank Max Alekseyev for help with the proof of Lemma \ref{L4.2} in the form of a Math Overflow answer.
\section{Dirac operators and canonical systems}
In this section, we discuss basic material on Dirac equations. Somewhat unusually perhaps, we give center stage to
the associated canonical systems.

Unlike the Dirac equation, such a system \eqref{can} does not have coefficient functions that do not get multiplied by the spectral parameter $z$,
so if we want to rewrite a Dirac equation \eqref{Dirac} as a canonical system, the obvious attempt is variation of constants about $z=0$.
This is also the standard procedure in analogous settings; compare \cite[Sections 1.3, 5.3, 7.1]{Rembook}.

So define the transfer matrix $T(x)$ for \eqref{Dirac} at $z=0$ as the matrix solution of $JT'+WT=0$ with the initial value $T(0)=1$ (the $2\times 2$
identity matrix). Observe that since we can write the equation as $T'=JWT$ and $\tr JW=0$, we have $\det T(x)=1$ for all $x$.

Next, if $y$ is any solution of \eqref{Dirac}, define $u(x)$ by writing $y=Tu$. A quick calculation shows that $u$ solves the canonical system \eqref{can}
with coefficient function
\begin{equation}
\label{2.1}
H(x) = T^t(x)T(x) .
\end{equation}
Note that, as announced above, we now work with canonical systems that are not trace normed; rather, $H$ satisfies $\det H(x)=1$.
We could change to the normalization $\tr H=1$ by a change of variable \cite[Section 1.3]{Rembook},
but it is more convenient to keep the determinant normalized instead.
We then still have uniqueness of coefficients for given spectral data in the sense that if $m_+(z;H_1)=m_+(z;H_2)$ and $\det H_1(x)=\det H_2(x)=1$,
then $H_1(x)=H_2(x)$ for almost every $x\ge 0$. (The determinant normalization has the potential drawback that if $\det H(x)=0$ on a set
of positive measure, then there is no equivalent $H_1$ with $\det H_1=1$, while any $H$ can be trace normed.)

The steps that led to \eqref{2.1} can be taken in reverse, so it is also true that if $u$ solves the canonical system \eqref{can} with coefficient function given by \eqref{2.1}
and we define $y=Tu$, then $y$ solves the Dirac equation \eqref{Dirac}, with $W=-JT'T^{-1}$.
So we have succeeded in rewriting an arbitrary Dirac equation as a canonical system.
Let's summarize and add some precision and more detail to this picture.
\begin{Theorem}
\label{T2.1}
Let $W\in \mathcal W[0,\infty)$. Then there is a canonical system $H(x)$, $x\ge 0$, such that
\begin{equation}
\label{2.3}
m_+(z; W) = m_+(z; H), \quad z\in\C^+ .
\end{equation}
This canonical system is unique if normalized appropriately, which can most naturally be done by requiring that $\det H(x)=1$.
It is then given by \eqref{2.1}, and it satisfies
\begin{equation}
\label{2.2}
H\in AC[0,\infty), \quad H(0)=1, \quad \det H(x)=1 .
\end{equation}

Conversely, if an $H\in\mathcal C[0,\infty)$ satisfying \eqref{2.2} is given, then there is a Dirac potential $W\in \mathcal W [0,\infty)$ such that \eqref{2.3} holds.
\end{Theorem}
We have focused on half line problems here since this is the most basic case. We will eventually be interested in whole line problems, when we discuss
reflectionless operators, but the standard way to set up a spectral representation of these is to cut them into two half lines and use the half line $m$ functions
$m_{\pm}$ as the key ingredients. See \cite[Section 3.7]{Rembook}. Of course, the analog of Theorem \ref{T2.1} for left half lines also holds.

In the last part of the theorem, $W$ is not uniquely determined by its $m$ function or, equivalently, by $H$. This is well known, and we will discuss the matter
in more detail in Theorem \ref{T2.2} below.

As announced, Theorem \ref{T2.1} lets us think of Dirac operators as special canonical systems, and we will occasionally be a bit cavalier about the distinction
and talk about Dirac potentials $W\in\mathcal W$ as if they were literally canonical systems when we really mean the $H\in\mathcal C$ that is associated with $W$.
\begin{proof}
Everything in the first part is an immediate consequence of the simple transformation $y=Tu$ that we discussed above. To prove \eqref{2.3}, notice that by \eqref{2.1},
we have $y^*y=u^*Hu$, so $u\in L^2_H(0,\infty)$ if and only if $y\in L^2((0,\infty);\C^2)$. So if $z\in\C^+$ and $y\in L^2$ solves \eqref{Dirac}, then
\[
m_+(z;W) = y(0) = T(0)u(0) = u(0) =m_+(z;H) ,
\]
as claimed. (It is probably also helpful in this context to be aware of the fact formulated
as Corollary \ref{C2.1} below.) The uniqueness of a suitably normalized canonical system with prescribed $m$ function is one of the basic (but difficult) results
on these, and it was already mentioned above; compare \cite[Theorem 5.1]{Rembook}.
The properties of $H=T^tT$ stated in \eqref{2.2} are obvious because $T\in AC$, $\det T=1$, $T(0)=1$.

Now suppose that an $H$ satisfying \eqref{2.2} is given. Let's write
\[
H(x) = \begin{pmatrix} R_1(x)^2 & R_1(x)R_2(x)\cos\theta(x) \\ R_1(x)R_2(x)\cos\theta & R_2(x)^2 \end{pmatrix} ,
\]
with $R_j,\theta\in AC$, $R_j(0)=1$, $\theta(0)=\pi/2$, $R_j(x)> 0$. This is possible at every fixed $x$ because any matrix $H(x)> 0$
can be written in this form. That the additional properties of the \textit{functions }$R_j,\theta$ may be insisted on
follows from the initial value $H(0)=1$ and the fact that $H\in AC$.
From the further condition $\det H=1$ we then see that $R_1^2R_2^2\sin^2\theta = 1$ and in fact
\begin{equation}
\label{2.4}
R_1(x)R_2(x)\sin\theta (x) = 1
\end{equation}
since this (and not $-1$) is the value at $x=0$ and the function on the left-hand side is continuous.

Now let
\[
T(x) = \begin{pmatrix} R_1(x) & R_2(x)\cos\theta(x) \\ 0 & R_2(x)\sin\theta(x) \end{pmatrix} .
\]
It is clear that $T(0)=1$, $T\in AC$, $\det T=1$, and, by a quick calculation, $H=T^tT$.

We now define, as we must, $W(x)=-JT'(x)T(x)^{-1}$. This is clearly locally integrable on $0\le x<\infty$, but the symmetry of $W(x)$ is perhaps not obvious.
However, from a calculation we find $W_{21}-W_{12}= (R_1R_2\sin\theta)'$, and this equals $0$, as required, by \eqref{2.4}.
So $W\in\mathcal W$, and by construction we have $JT'+WT=0$, $H=T^tT$. So the already established first part now shows that \eqref{2.3} holds,
as desired.
\end{proof}
\begin{Corollary}
\label{C2.1}
For any $W\in\mathcal W[0,\infty)$, $z\in\C^+$, the Dirac equation \eqref{Dirac} has exactly one linearly independent solution $y\in L^2(0,\infty)$.
\end{Corollary}
This is usually stated as saying that a Dirac equation is in the limit point case at $\infty$. Given the connection to canonical systems, the proof
is delightfully simple now, especially when compared with the traditional treatment \cite[Chapter 8]{LevSar}.
\begin{proof}
We already observed at the beginning of the proof of Theorem \ref{T2.1}
that the square integrable solutions of a Dirac equation are obtained exactly from the $L^2_H$ solutions of the associated
canonical system, so the statement is equivalent to the corresponding claim for the canonical system. A canonical system is in the limit point case
if and only if $\tr H\notin L^1$ \cite{Arch}, \cite[Theorem 3.5]{Rembook}. Since $\det H=1$ and $H\ge 0$, we have $\tr H\ge 2$.
\end{proof}
We also make a definition that summarizes the message of Theorem \ref{T2.1}, for whole line operators.
\begin{Definition}
\label{D2.1}
We define the collection of \textit{Dirac canonical systems }as
\[
\mathcal D = \{ H\in\mathcal C (\R) : H\in AC(\R), H(0)=1, \det H(x)=1 \} .
\]
\end{Definition}
Again, we have abandoned the original normalization $\tr H=1$, but that was only convenient, not necessary.
An equivalent version of the definition declares an $H\in\mathcal C$, $\tr H=1$, to be in $\mathcal D$ if and only if $H\in AC$, $H(0)=1/2$,
$\det H(x)>0$.

We now discuss transformations of $W\in\mathcal W$ that preserve the $m$ function. These are obtained by letting the following group of functions act
on $\mathcal W[0,\infty)$:
\[
\mathcal A_0=\{ \alpha\in AC[0,\infty): \alpha(0)=0\} .
\]
The subscript zero reminds us of the requirement $\alpha (0)=0$. We will drop this condition later, when we briefly discuss the action of a more general
(non-abelian) group at the end of this section.

We use the notation
\[
R_{\beta} = \begin{pmatrix} \cos\beta & -\sin\beta \\ \sin\beta & \cos\beta \end{pmatrix}
\]
to refer to a general rotation matrix $R\in\operatorname{SO}(2)$.
For $W\in\mathcal W$ and $\alpha\in\mathcal A_0$, define
\begin{equation}
\label{2.6}
(\alpha\cdot W)(x) = R_{\alpha(x)}W(x)R^t_{\alpha(x)} + \alpha'(x) .
\end{equation}
It is clear that $\alpha\cdot W\in\mathcal W$ again.
The notation emphasizes the fact that this is a group action of $\mathcal A_0$ on $\mathcal W$, if we use pointwise addition as the group operation
on $\mathcal A_0$.
We will use this dot notation for group actions consistently in the sequel also.
\begin{Theorem}
\label{T2.2}
Let $W_1$, $W_2\in\mathcal W[0,\infty)$.
We have
\[
m_+(z; W_1) = m_+(z;W_2) \quad (z\in\C^+)
\]
if and only if $W_2=\alpha\cdot W_1$ for some $\alpha\in\mathcal A_0$.
In this case, $\alpha\in\mathcal A_0$ is unique, and $T_2(x)=R_{\alpha(x)}T_1(x)$.
\end{Theorem}
\begin{proof}
We begin with the uniqueness claim about $\alpha$. Since we have a group action, it suffices to check that $\alpha\cdot W=W$ only if $\alpha=0$.
This follows at once by taking the trace of both sides of $R_{\alpha}WR^t_{\alpha}+\alpha' = W$ to conclude that $\alpha'=0$.

Now assume that $W_2=\alpha\cdot W_1$. By Theorem \ref{T2.1}, we can establish the claim by showing that the associated canonical systems
satisfy $H_1=H_2$. This will be clear as soon as we have the asserted identity $T_2=R_{\alpha}T_1$, and this is obvious since an easy calculation
shows that if $JT'_1+W_1T_1=0$, then $J(R_{\alpha}T_1)'+(\alpha\cdot W_1)R_{\alpha}T_1=0$. So since also $(R_{\alpha}T_1)(0)=1$,
this matrix function solves the same initial value problem as $T_2$ and thus agrees with this transfer matrix.

Conversely, suppose now that the $m$ functions of $W_1$, $W_2$ agree. By Theorem \ref{T2.1} and the uniqueness of suitably normalized canonical systems
with given $m$ function, this means that $H_1=H_2$. In other words,
the transfer matrices of the Dirac systems satisfy
\begin{equation}
\label{2.7}
T^t_2(x)T_2(x)=T^t_1(x)T_1(x) .
\end{equation}
Now consider the polar decompositions $T = U|T|$ for $T=T_1(x),T_2(x)$.
Since $T$ has real entries, so does $|T| = (T^*T)^{1/2}=(T^tT)^{1/2}$. Moreover, $\det |T|=|\det T|=1$, so $U$ is also real and $\det U=1$.
In other words, $U\in SO(2)$, that is, $U=R_{\beta}$ for some $\beta$. Finally, \eqref{2.7} shows that $|T_1(x)|=|T_2(x)|$.

These observations imply that
$T_2(x) = R_{\alpha(x)} T_1(x)$, for some function $\alpha(x)$.
Since $T_2,T^{-1}_1\in AC$, we may also assume that $\alpha\in AC$ here. Furthermore, since $T_2(0)=T_1(0) (=1)$,
we can take $\alpha (0)=0$. So $\alpha\in\mathcal A_0$. Now the calculation that we already did above (though not explicitly) will confirm
that $T_2$ solves $JT'_2+(\alpha\cdot W_1)T_2=0$, so $W_2=-JT'_2T^{-1}_2=\alpha\cdot W_1$, as claimed.
\end{proof}
A sensible way to take advantage of Theorem \ref{T2.1} would be to restrict one's attention to potentials $W$ satisfying $\tr W(x)=0$.
These potentials realize exactly all the $m$ functions one can get from Dirac systems from $\mathcal D$. The following different normalization
will be convenient for us later.
\begin{Proposition}
\label{P2.1}
For any $W\in\mathcal W (\mathbb R)$, there is an $\alpha\in\mathcal A_0 (\R)$ such that $(\alpha\cdot W)_{11}\equiv 0$.
\end{Proposition}
\begin{proof}
We in fact did this already: the potential $W$ constructed in the last part of the proof of Theorem \ref{T2.1} has the desired property.

But we can also easily do it from scratch:
By calculating the matrix element in question, we see that we are trying to find an $\alpha$ such that
\[
\alpha' + W_{11}\cos^2\alpha + W_{22}\sin^2\alpha - 2W_{12}\sin\alpha\cos\alpha = 0, \quad \alpha(0)=0 .
\]
This initial value problem for the ODE $\alpha' = f(x,\alpha)$ has a global solution since $f$ satisfies the uniform (in $\alpha$)
bound $|f(x,\alpha)|\lesssim \|W(x)\|$.
\end{proof}
The off-diagonal elements of $W$ can not always be eliminated in this fashion since we would now obtain $\alpha$ from
an algebraic equation, and this function is not guaranteed to be absolutely continuous, not even if $W$ is smooth.

We now make a few quick remarks on general boundary conditions of \eqref{Dirac}. On a half line, this equation generates a one-parameter
family of self-adjoint operators on $L^2((0,\infty);\C^2)$, corresponding to the different boundary conditions
\begin{equation}
\label{bc}
y_1(0)\sin\theta + y_2(0)\cos\theta = 0
\end{equation}
that can be imposed on elements of the domain. Our definition of the half line $m$ function as $m_+(z)=f_+(0,z)$ means that specifically
the boundary condition $y_2(0)=0$ was (tacitly) chosen. One can now wonder what additional spectral data might have been realized by other boundary conditions,
and the answer is none at all. The situation is exactly the same as for canonical systems: one boundary condition is enough because a change of boundary condition
can be simulated by transforming $W$ (or $H$) instead and keeping the boundary condition the same. Please see \cite[Section 3.6]{Rembook} for a detailed
treatment of this issue for canonical systems.

Let's now briefly look at the details for Dirac operators. First of all, a piece of notation: we write
\begin{equation}
\label{2.57}
\begin{pmatrix} a & b \\ c&d \end{pmatrix} \cdot z = \frac{az+b}{cz+d}
\end{equation}
for the action of an invertible matrix on $\C_{\infty}$ as a linear fractional transformation. Recall also that the automorphisms
of $\C^+$ are obtained from the matrices $A\in\SL (2,\R)$.

If we now change boundary conditions to \eqref{bc},
then the new $m$ function $m_{\theta}$ is related to the old one by \cite[Section 8.2]{LevSar}
\[
m_{\theta}(z) = \frac{m(z)\cos\theta - \sin\theta}{m(z)\sin\theta+\cos\theta}= R_{\theta}\cdot m(z)  .
\]
Acting in this way by a linear fractional transformation on the $m$ function will transform the coefficient function of a canonical system
as follows \cite[Theorem 3.20]{Rembook}:
\[
H_{\theta}(x) = R^{-1t}_{\theta} H(x) R^{-1}_{\theta} = R_{\theta} H(x) R^t_{\theta}
\]
In particular, in the situation above, when we apply this transformation to an $H\in\mathcal D$, we see that $H_{\theta}\in\mathcal D$ also,
so indeed all Dirac $m$ functions are already obtained from one boundary condition only. We can be more explicit: if we similarly introduce
$T_{\theta}(x) = R_{\theta}T(x)R^t_{\theta}$, with $T$ denoting the transfer matrix of $W$, then this solves
$JT'_{\theta}+R_{\theta}WR^t_{\theta}T_{\theta}=0$, and obviously $T_{\theta}(0)=1$. Since $T^t_{\theta}T_{\theta}=H_{\theta}$,
this shows that one way of realizing the $m$ function $m_{\theta}$ for boundary condition \eqref{bc} is to change the potential to
\begin{equation}
\label{2.16}
W_{\theta}(x)=R_{\theta}W(x)R^t_{\theta}
\end{equation}
instead and leave the boundary condition $y_2(0)=0$ unchanged. Of course, at the level of the
Dirac potentials, there are other options: by Theorem \ref{T2.2}, exactly the potentials $\alpha\cdot W_{\theta}$ work.

We can now develop a more conceptual view of what exactly transformations like the one from Proposition \ref{P2.1} are achieving. We combine
three ingredients: the actions of $\mathcal A_0$ and $\operatorname{PSO}(2)=\operatorname{SO}(2)/\{ \pm 1\}$
on $\mathcal W$, with the latter given by \eqref{2.16}, and, thirdly,
the action of $\mathbb R$ by shifts $(t\cdot W)(x)=W(t+x)$.

We can very conveniently incorporate \eqref{2.16} into the first group action by simply dropping the requirement that $\alpha(0)=0$. To avoid spurious
stabilizers of the group action \eqref{2.6}, we now interpret group elements as functions taking values in the unit circle $S=\{|z|=1\}$. Put differently,
we identify $\alpha,\beta$ if $\beta(x)-\alpha(x) \equiv n\pi$. (This was unnecessary initially since the condition $\alpha(0)=0$, which we have now dropped, singled out
a unique representative.) So we define
\[
\mathcal A = \{ f(x)=e^{2i\alpha(x)}:\R\to S : \alpha\in AC(\mathbb R) \} ,
\]
and of course the group operation is now pointwise multiplication of functions $f,g\in\mathcal A$.
Having made this definition, we will find it more convenient, though, to continue to work with the angle functions $\alpha$, and we will be quite
cavalier about the distinction between $\alpha$ and $f=e^{2i\alpha}$.

The larger group $\mathcal A$ still acts on $\mathcal W$ by \eqref{2.6}.
Of course, this action no longer preserves $m$ functions; rather, since we can write $\alpha=\beta+\alpha(0)$,
with $\beta\in\mathcal A_0$, we have
\begin{equation}
\label{2.17}
\pm m_{\pm}(z; \alpha\cdot W)=R_{\alpha(0)}\cdot (\pm m_{\pm}(z;W)) .
\end{equation}
This is the same as changing boundary conditions by $\theta=\alpha(0)$. The new \textit{whole line }operator $\alpha\cdot W$ is still unitarily equivalent to $W$,
and the property of being reflectionless on a set is preserved. See \cite[Theorems 7.2, 7.9(a)]{Rembook}; in fact, these results are considerably more general,
and our more specific claims
are immediately plausible without any theory since for the whole line problem, the boundary condition at $x=0$ has no intrinsic significance. Rather, it is an artifact
that tells us how exactly we are going to construct a spectral
representation of the whole line operator via $m_{\pm}$, which do depend on the boundary condition. We also mention in passing
that the singular spectra of the \textit{half line }problems can change
dramatically. These phenomena are well known and they have been studied extensively; see, for example, \cite{Simr1}.

We also see from \eqref{2.17} that even if $\alpha\in\mathcal A_0$, then the transformation from $W$ to $\alpha\cdot W$ will have an effect on
the $m$ functions of the \textit{shifted }problems. More precisely, we have
\[
\pm m_{\pm}(z; (\alpha\cdot W)(t+x)) = R_{\alpha(t)}\cdot (\pm m_{\pm}(z; W(t+x))) .
\]
So while the transformation by an $\alpha\in\mathcal A_0$ does not do anything to the original $m$ functions,
it does allow us to vary the $m$ function along the orbit of shifts.

As announced, we can capture all this in an extended abstract framework.
We add shifts to our group actions and thus build a corresponding larger group as the semidirect product $\mathcal A\rtimes\R$, with $\R$ acting
on $\mathcal A$ by shifts also: $(t\cdot\alpha)(x)=\alpha(t+x)$. Obviously, these maps are automorphisms of $\mathcal A$, so the semidirect product
is well defined. Its elements are $(\alpha, t)$, and the group operation is given by
\[
(\alpha, s)(\beta, t) = (\alpha+ s\cdot\beta, s+t) .
\]
\begin{Theorem}
\label{T2.3}
The following combination of the actions \eqref{2.6}, \eqref{2.16}, and shifts defines an action of the (non-abelian) group $\mathcal A\rtimes\R$ on $\mathcal W$:
\[
((\alpha,t)\cdot W)(x) = R_{\alpha(x)}W(t+x)R^t_{\alpha(x)} + \alpha'(x)
\]
The $m$ functions obey
\begin{equation}
\label{2.19}
\pm m_{\pm}(z; (t\cdot(\alpha\cdot W)) = R_{\alpha(t)}\cdot (\pm m_{\pm}(z; t\cdot W)) .
\end{equation}
\end{Theorem}
The \textit{proof }consists of a direct verification of these claims, which we leave to the reader.
Our notation in \eqref{2.19} is a bit sloppy, but more intuitive and easier on the eye
than the more precise notation $(0,t)$, $(\alpha,0)$ for the group elements that we have written simply as $t$ and $\alpha$, respectively.

The theme of \eqref{2.19} could be much expanded: the rotations $R_{\alpha(t)}$ can be combined with the transfer matrix $T(t,z;W)$ of
\eqref{Dirac} (not the one we used above, but the more general version with $z\in\C$ restored)
to provide a cocycle for the action of $\mathcal A\rtimes\R$ that updates the $m$ functions along orbits.
See \cite{RemToda} for much more on this topic in a slightly different setting.
\section{The space $\mathcal R(U)$}
In this section, we study $\mathcal R(U)$, for $U=(-\infty,c)\cup (d,\infty)$, with $d-c>0$. We will be especially interested
in the Dirac operators in this set.

We will explicitly discuss
only the case $c=-1,d=1$. The general case can easily be reduced to this situation by using the transformed potentials
(in the Dirac case) $W_a(x)=W(x)+a$, $a\in\R$, and
$W_g(x) = gW(gx)$, $g>0$, which implement shifts $z\mapsto z+a$ and rescalings $z\mapsto gz$, respectively, of the spectral parameter $z$.

We now follow our outline from the introduction. Let $H\in\mathcal R(U)$, with
\[
U=(-\infty,-1)\cup (1,\infty) .
\]
We then consider the holomorphic function $M:\Omega\to\overline{\C^+}$ from Lemma \ref{L1.1}. By Lemma \ref{L1.2}, or rather its proof,
we can give $M$ the domain
$\Omega=\C_{\infty}\setminus [-1,1]$. Note that, unlike the original domain $\Omega\setminus\{\infty \}$, this domain $\Omega$ is simply connected,
so is conformally equivalent to $\C^+$.
A convenient explicit conformal map is given by
\begin{equation}
\label{phi}
\varphi: \C^+\to\Omega , \quad \varphi(\lambda) =\frac{2\lambda}{\lambda^2+1} .
\end{equation}
This maps the semicircle $\{z\in\C^+: |z|=1\}$ to $U\cup\{\infty \}$, the semidisk $\{z\in\C^+: |z|<1\}$ goes to $\C^+$, and its exterior is mapped to $\C^-$.

Finally, we introduce the key object, namely the new Herglotz function
\[
F: \C^+\to\overline{\C^+}, \quad F(\lambda) = M(\varphi(\lambda)) ;
\]
we call this the $F$ \textit{function }of $H\in\mathcal R (U)$. We will denote the collection of all Herglotz functions by $\mathcal F$.
So the elements $G\in\mathcal F$ are exactly the holomorphic functions $G:\C^+\to\overline{\C^+}$.
\begin{Theorem}
\label{T3.1}
This map $\mathcal R (U)\to\mathcal F$, $H\mapsto F(\lambda ;H)$, that sends an $H\in\mathcal R(U)$ to its $F$ function, is a homeomorphism.
\end{Theorem}
For our purposes here, the important fact is that we have a bijection and thus a parametrization of the $H\in\mathcal R(U)$.
However, the additional claim about the continuity of the map and its inverse is
easy to establish, and it may be of interest in other investigations, so we have stated it too. See also \cite{MykSush2}.

We must explain what topologies we are using.
In fact, both spaces come equipped with natural metrics. On $\mathcal R(U)\subseteq \mathcal C(\R)$ we use the metric on general canonical systems
that is discussed in detail in \cite[Section 5.2]{Rembook}. More importantly for us here, the correspondence $H\leftrightarrow (m_-,m_+)$
between a canonical system and its pair of $m$ functions becomes a homeomorphism between $\mathcal C(\R)$ and $\mathcal F^2$,
provided we give $\mathcal F$ the topology we are about to discuss next. More accurately, we will use the following metric on $\mathcal F$:
\[
d(G,H)=\max_{|z-2i|\le 1} \delta (G(z), H(z)) .
\]
Here, $\delta$ denotes the spherical metric, and $\overline{\C^+}\subseteq\C_{\infty}$ is thought of as a subset of the Riemann sphere
$\C_{\infty}\cong S^2$. Though this may not be immediately apparent from the appearance of $d$, convergence in this metric is the same
as locally uniform convergence of the Herglotz functions. The space
$(\mathcal F,d)$ is compact. For more details on these facts, please consult again \cite[Section 5.2]{Rembook}.

Finally, since \textit{reflectionless }canonical systems are determined by just one half line \cite[Theorem 7.9(b)]{Rembook} and the reconstruction
of the other half line defines a continuous map, we can also measure the distance between $H_1,H_2\in\mathcal R(U)$ by just
$d(m^{(1)}_+,m^{(2)}_+)$, ignoring the left half lines, and we still obtain the same topology.
\begin{proof}
Since we can extract the $m$ functions $m_{\pm}$ from the $F$ function, the map $H\mapsto F$ is clearly injective. It is also surjective because
we can obtain an arbitrary $G\in\mathcal F$ as the $F$ function of an $H$ by simply reversing the steps that led from $H$ to $F$. More explicitly,
given a $G\in\mathcal F$, define
\begin{equation}
\label{3.1}
m_+(z)=G(\varphi^{-1}(z)), \quad m_-(z) =-\overline{G(\varphi^{-1}(\overline{z}))}; \quad z\in\C^+ .
\end{equation}
Clearly, $m_{\pm}\in\mathcal F$, so there is a (unique) $H\in\mathcal C(\R)$ that has $m_{\pm}$ as its $m$ functions.

It is also clear, by construction, that $F(\lambda;H)=G(\lambda)$, or at least it would be if we already knew that $H\in\mathcal R(U)$.
It remains to verify this last claim. The limits
\[
m_{\pm}(x)\equiv \lim_{y\to 0+} m_{\pm}(x+iy)
\]
exist for all $x\in U$ and since $x=\overline{x}\in\Omega$, they agree with what we obtain
by simply plugging $z=x$ into \eqref{3.1}. In particular, $m_+(x)=-\overline{m_-(x)}$ for all $x\in U$, as required.

As briefly discussed above, for any Herglotz functions $G_n,G$, we have $G_n\to G$ locally uniformly on $\C^+$ if and only if $G_n(z)\to G(z)$ uniformly
on $z\in K$ for a fixed compact set $K\subseteq\C^+$ with an accumulation point. In particular, it is now obvious that $F_n\to F$ if and only if
$m^{(n)}_{\pm}\to m_{\pm}$, and this establishes the additional claims about the continuity of the map $H\mapsto F$ and its inverse.
\end{proof}
\begin{Theorem}
\label{T3.2}
Let $H\in\mathcal R(U)$. Then $H\in\mathcal D$ if and only if $F(i;H)=i$.
\end{Theorem}
We are here dealing with an inverse spectral problem: we want to decide whether or not given spectral data (here, an $F$ function) come from
a Dirac operator. These questions are classical, at least in the version for other spectral data such as spectral measures or $m$ functions.
The foundational paper is \cite{GasLev}; see also \cite[Chapter 12]{LevSar} for a presentation of this work.

The whole area quickly becomes very technical. What is worse, to the best of our knowledge, the characterization of the spectral measures (or $m$ functions)
of Dirac operators with \textit{general }potentials $W\in\mathcal W$ has still not been established fully rigorously yet, even though it is clear what the expected answer is.
\cite{GasLev, LevSar} assume continuity of the potential; \cite{ZengPhD} can handle the case $\phi\in L^2_{\textrm{loc}}$, with $\phi$ being an auxiliary
function that also plays a major role in the Gelfand-Levitan theory \cite{GasLev,HarAsh,LevSar}, but this is again not quite the general situation, when
$\phi\in L^1_{\textrm{loc}}$.

Fortunately for us here, we can bypass these thorny issues almost completely since we are interested in $H\in\mathcal R (U)$, which is a strong additional
assumption, and thus a much simplified version of inverse spectral theory is sufficient for our purposes. Still, we will be content with only sketching the proof
of Theorem \ref{T3.2}, and we refer the reader to the sources mentioned above and especially \cite[Section 6.4]{Rembook} and \cite{ZengPhD} for further details.
\begin{proof}
One direction can be settled comfortably by recalling that Dirac $m$ functions satisfy $m(z)\to i$ when $z\to\infty$ along suitable paths \cite{EHS,HarAsh}.
Since $\varphi(i)=\infty$, this means that the $F$ function of an $H\in\mathcal D\cap\mathcal R(U)$ must satisfy $F(i)=i$.

Conversely, suppose now that this condition holds. This implies that $m_{\pm}(x) = i + O(1/|x|)$ for $x\in U$. Let's focus on $m_+$, for convenience,
but of course an analogous argument will apply to $m_-$. We have the Herglotz representation
\begin{equation}
\label{3.4}
m_+(z) = a + \int_{-\infty}^{\infty} \left( \frac{1}{t-z}-\frac{t}{t^2+1}\right) \, d\rho(t) .
\end{equation}
The representing measure $\rho$, which is a spectral measure for the half line operator, is purely absolutely continuous on $U$ since $m$ is holomorphic there,
and its density satisfies $(1/\pi)\Im m(x)=1/\pi +O(1/x)$. Since a compactly supported measure has an entire Fourier transform, we conclude that
\[
\phi \equiv (d\rho(x) - dx/\pi)\,^{\widehat{\:}} \in L^2_{\textrm{loc}} ,
\]
and this is the sufficient condition from \cite{ZengPhD} that was already briefly mentioned above. (We could extract more information about
$\phi$ from what we are given about the asymptotics of $m_+$, and thus we do not really need the full force of this result.)

We conclude that $\rho$ is the spectral measure of some Dirac equation on the half line $x\ge 0$. Since $\rho$ determines $m_+$ up to the
additive real constant $a$ from \eqref{3.4}, it follows that for suitable $b\in\R$, the Herglotz function $m_+(z)+b$ will be the $m$ function
of that same Dirac equation. But we already have $m_+(\infty)=i$, and adding a $b\not= 0$ would lead to a function that no longer satisfies this
necessary condition, so $m_+$ itself is a Dirac $m$ function.
\end{proof}
We can state much of what Theorem \ref{T3.2} says in more abstract style. This will also add some
perspective to concrete results we will prove later.

We introduce
\[
\mathcal R_0(U) =\{ H\in\mathcal R(U): \sigma(H)\subseteq \overline{U} \} .
\]
It will also be convenient to have a more succinct notation available for the Dirac operators in these spaces.
\begin{Definition}
\label{D3.1}
We write $\mathcal D(U)=\mathcal D\cap\mathcal R(U)$, $\mathcal D_0(U)=\mathcal D\cap\mathcal R_0(U)$.
\end{Definition}
Of course, $\mathcal D_0(U),\mathcal R_0(U)$ are much smaller spaces than $\mathcal D(U), \mathcal R(U)$.
Especially for Jacobi and Schr{\"o}dinger equations, such operators
have been studied extensively \cite{GesHol1,GesHol2,Mum,Teschl};
they are usually called \textit{finite gap }operators, referring to the structure of the spectrum in the more general
situation when $\overline{U}$ is replaced by a closed set with finitely many gaps.

$\mathcal D_0(U)$ is a circle topologically, and $\mathcal R_0(U)$ itself is homeomorphic to $S^3$;
in the latter case, if the trivial canonical systems with constant real $m$ functions are thrown out, then we obtain a (non-compact)
solid torus $\mathcal R_0(U)\setminus\mathcal Z\cong D\times S^1$. Please see \cite{FR} for much more on this topic.
\begin{Theorem}
\label{T3.3}
$\{ F(\lambda;H): H\in\mathcal D(U)\}\subseteq\mathcal F$ is a compact convex set.
Its extreme points are
\[
\{F(\lambda;H): H\in\mathcal D_0(U) \}.
\]
In the $\tr W=0$ normalization,
these correspond exactly to the constant Dirac potentials
\begin{equation}
\label{3.21}
W_{\beta}(x) = \begin{pmatrix}\sin \beta & \cos \beta\\
\cos \beta & -\sin \beta \end{pmatrix} , \quad 0\le\beta <2\pi .
\end{equation}
\end{Theorem}
We need both defining properties of the $H\in\mathcal D(U)=\mathcal D\cap\mathcal R(U)$ here: $\mathcal D$ is convex, but not compact, while $\mathcal R(U)$
is compact, but not convex in an obvious way (the vector space operations we would like to use do not extend to $F\equiv\infty$).
\begin{proof}
Since $\mathcal F$ itself is compact,
the set $\{F\in\mathcal F: F(i)=i\}$ is obviously compact and convex. To find its extreme points, recall the Herglotz representation formula, in the version
\begin{equation}
\label{3.5}
F(\lambda) = a + \int_{\R_{\infty}} \frac{1+t\lambda}{t-\lambda}\, d\nu(t) ,
\end{equation}
with $a\in\R$, and $\nu$ is a finite Borel measure on the compact space $\R_{\infty}=\R\cup\{\infty \}$.
Observe that
\[
\frac{1+t\lambda}{t-\lambda} = \frac{1}{\sqrt{t^2+1}} \begin{pmatrix} t & 1 \\ -1 & t \end{pmatrix} \cdot\lambda ,
\]
and these matrices range precisely over $\operatorname{PSO}(2)=\operatorname{SO}(2)/\{ \pm 1\}$ when $t$ varies over $\R_{\infty}$.
Thus we obtain a third and particularly elegant version of the Herglotz representation formula:
\[
F(\lambda) = a + \int_{[0,\pi)} R_{\beta}\cdot\lambda\, d\mu(\beta) .
\]
Since $F(i)=a+i\nu(\R_{\infty})=a+i\mu([0,\pi))$, the Herglotz functions with $F(i)=i$ are exactly those of the form
\begin{equation}
\label{3.7}
F(\lambda) = \int_{[0,\pi)} R_{\beta}\cdot\lambda\, d\mu(\beta) , \quad\mu([0,\pi))=1 .
\end{equation}
It is now obvious that the extreme points correspond exactly to the measures $\mu=\delta_{\theta}$, so are given by the $F$
functions $F(\lambda)=R_{\theta}\cdot\lambda$.

For constant Dirac potentials, the Dirac equation can be solved explicitly and everything can be worked out.
So it is now straightforward to confirm that the potentials listed in \eqref{3.21} indeed give exactly the $F$ functions $F=R_{\theta}\cdot\lambda$.

In fact, it suffices to do this calculation for (say)
\[
W(x)\equiv \begin{pmatrix} 0 & -1 \\ -1 & 0 \end{pmatrix} .
\]
We find that $F(\lambda)=\lambda$,
corresponding to $\mu=\delta_0$ in \eqref{3.7}. The remaining functions $R_{\theta}\cdot\lambda$
can then be obtained by acting on this $F$ by the rotations $R_{\theta}$. Compare our discussion in Section 2.
In particular, this material shows that $F(\lambda)=R_{\theta}\cdot \lambda$ is the $F$ function of the (constant) Dirac potential
\[
W(x) = R_{\theta}\begin{pmatrix} 0 & -1 \\ -1 & 0 \end{pmatrix} R^t_{\theta} = \begin{pmatrix} \sin 2\theta & -\cos 2\theta\\
-\cos 2\theta & -\sin 2\theta \end{pmatrix} .
\]

We must still show that $\{ W_{\beta}\}=\mathcal D_0(U)$, and we sketch two possible ways of doing this.

The first argument is quick, but relies on a standard method of parametrizing the operators from $\mathcal D_0(U)$, which we didn't discuss here. See, for example,
\cite[Sections 2, 3]{FR} for this material. A comparison of the $m$ functions obtained in this way with the $F$ functions $F(\lambda)=R_{\theta}\cdot\lambda$
obtained above will then confirm that we are indeed looking at exactly the elements of $\mathcal D_0(U)$.
This last step will require some stamina if done by a direct brute force calculation. However, we can also match only $F(\lambda)=\lambda$, say,
with an individual element of $\mathcal D_0(U)$, and then we verify that $\{R_{\theta}\}=\operatorname{SO}(2)$ acts
transitively on $\mathcal D_0(U)$ to obtain the desired conclusion more elegantly.

Alternatively, one can show by hand that if $H\in\mathcal D(U)$, then spectrum in $(-1,1)$ is avoided precisely if
$\mu=\delta_{\theta}$ or, equivalently, $F=R_{\theta}\cdot\lambda$. Here, the first step would be to recall that the measures
of $m_{\pm}$ must be discrete on this set, or else there would be essential spectrum there. Then possible discrete eigenvalues can be located
using the values of $m_{\pm}$ on this interval, in completely elementary (if tedious) fashion, by using the criterion $m_+=m_-$.
\end{proof}
It is perhaps interesting to note that Theorem \ref{T3.2} can be quickly recovered from the first part of Theorem \ref{T3.3}.
Indeed, as we just discussed, it is easy to confirm directly that $F=R_{\theta}\cdot\lambda$ is the $F$ function of an
$H\in\mathcal D(U)$. So if we also know that this latter space is compact and convex,
then \eqref{3.7} shows at once that it contains all $F\in\mathcal F$ with $F(i)=i$.
Note that this argument completely avoided any explicit reference to the machinery of inverse spectral theory.

Strictly speaking, it has established only one half of Theorem \ref{T3.2}, but the other half was much easier.
Moreover, that part, too, could be deduced from Theorem \ref{T3.3} in abstract style if we combine the identification
of the extreme points of $\mathcal D(U)$ with Choquet theory \cite{Phe} and \eqref{3.7}.
\section{Estimates on reflectionless Dirac potentials}
We now analyze in detail the Dirac potentials $W\in\mathcal D(U)$. The method we are going to use
goes back to Marchenko \cite{Mar}, and it was also presented in detail in \cite{HMcBR}. See also \cite{MykSush1}.
\begin{Theorem}
\label{T4.1}
Let $H\in\mathcal D(U)$, $U=\R\setminus [-1,1]$. We use the $\tr W=0$ normalization for the associated Dirac potentials, so we write
\[
W(x) = \begin{pmatrix} a(x) & b(x) \\ b(x) & -a(x) \end{pmatrix} ,
\]
with $a,b\in L^1_{\operatorname{loc}}(\R)$ uniquely determined by $H$ or the $F$ function of $H$.

Then $a,b$ are real analytic, and these functions and all their derivatives are bounded on $\R$.
They have holomorphic continuations $a(z), b(z)$ to a strip $\{|\Im z|< h\}$. Moreover,
\[
\|W(x)\| = \sqrt{a^2(x)+b^2(x)} \le 1 ,
\]
and equality $\|W(x_0)\|=1$ at a single $x_0\in\R$ holds if and only if
\[
W(x) \equiv \begin{pmatrix} \sin\beta & \cos\beta\\ \cos\beta & -\sin\beta \end{pmatrix}
\]
for some $\beta\in [0, 2\pi)$.

Finally, we have the formula
\begin{equation}
\label{4.9}
b(0)+ia(0)=-F'(i;H) .
\end{equation}
\end{Theorem}
We just encountered these special constant Dirac potentials before, in \eqref{3.21}. As we discussed in the context of Theorem \ref{T3.3},
the collection of these is exactly $\mathcal D_0(U)$, and the corresponding $F$ functions are the extreme points of
$\{F(\lambda; H): H\in\mathcal D(U)\}$. Now the various pieces
fit together neatly: $b(0)+ia(0)$ is a linear functional of $F$, by \eqref{4.9}, so will assume its maximum at the extreme points of the compact convex set $\{ F\}$.
\begin{proof}
The part about $W$ being real analytic has been split off as Theorem \ref{T6.1} of the next section, and we will prove it there, in the more general
version for Dirac potentials satisfying Hypothesis \ref{H1.1}. (The reader has already seen the still more general version
for arbitrary canonical systems, as Theorem \ref{T1.2}; this we prove in Section 6.) It will also be convenient to discuss the boundedness
of the derivatives of $W$ in this setting.

Also, the result looks best in the $\tr W=0$ normalization, so we gave it in this form,
but the asymptotic analysis in Lemma \ref{L4.2} below will become slightly easier for the normalization
from Proposition \ref{P2.1}. So this is what we are going to use here, and we must then translate back to the $\tr W=0$ normalization to obtain the result as stated.

So, given $H\in\mathcal D(U)$, we will work with the (unique) associated Dirac potential of the form
\begin{equation}
\label{4.10}
W(x) = \begin{pmatrix} 0 & q(x) \\ q(x) & -2p(x) \end{pmatrix} ,
\end{equation}
with $p,q\in L^1_{\textrm{loc}}(\R)$.

A key role will be played by the asymptotics of $m_+(z;W)$ as $|z|\to\infty$. By the definition of the $F$ function, we have
\begin{equation}
\label{4.1}
m_+(\varphi(i+h)) = F(i+h) = i+\sum_{n=1}^{\infty} f_n h^n .
\end{equation}
Recall that $\varphi$ was given by \eqref{phi}, and $\varphi(i)=\infty$. The power series expansion from \eqref{4.1}
is valid for $|h|<1$. Its first coefficient $f_0=i$ has been identified with the help of Theorem \ref{T3.2}.
For the first equality of \eqref{4.1}, we also need to stay inside the semi-disk $|i+h|\le 1$; otherwise, $F$ would deliver the other half line $m$ function $m_-$.

We can obtain more precise information on the coefficients $f_n$ from the Herglotz representation of $F$. Since $F(i)=i$, it reads
\begin{equation}
\label{4.2}
F(\lambda) = \int_{\R_{\infty}} \frac{1+t\lambda}{t-\lambda}\, d\nu(t), \quad \nu(\R_{\infty})=1 .
\end{equation}
\begin{Lemma}
\label{L4.1}
We have $| f_n| \le 1$ for all $n\ge 1$. Moreover, $|f_1|=1$ if and only if
$\nu =\delta_{t_0}$ for some $t_0\in\mathbb R_{\infty}$. For $n\ge 2$, we have $|f_n|=1$ if and only if $\nu=\delta_0$.
\end{Lemma}
\begin{proof}[Proof of Lemma \ref{L4.1}]
We may differentiate \eqref{4.2} under the integral sign. This gives, for $n\ge 1$,
\[
f_n = \frac{ F^{(n)}(i)}{n!} = \int_{\R_{\infty}} \frac{1+t^2}{(t-i)^{n+1}} \, d\nu(t) .
\]
All claims are immediate consequences of this formula.
\end{proof}

We now consider these $m$ functions $m(t;z)\equiv m_+(z; t\cdot W)$ for the shifted potentials
$(t\cdot W)(x)=W(t+x)$ also. (To be completely clear on this point, what we shift is the renormalized
potential with $W_{11}\equiv 0$. This is not the same as first shifting the $\tr W=0$ version from the statement of Theorem \ref{T4.1}
and then changing the normalization; in particular, this second version would have produced a different orbit $\{ m(t;z): t\in\R\}$. See again
Theorem \ref{T2.3} for context and recall that the acting group is non-abelian.)

We must also remember that $\mathcal D(U)$ is invariant under shifts, so $t\cdot W\in\mathcal D(U)$ again \cite[Theorem 7.9(a)]{Rembook}.

Clearly, $m(t;z)=f(t,z)=f_1(t,z)/f_2(t,z)$, with $f=y$ still denoting the unique, up to a factor, solution of \eqref{Dirac} that is square
integrable on right half lines. In particular, $m$ is an absolutely continuous function of $t$ for fixed $z\in\C^+$. It solves the Riccati equation
\[
\frac{dm(t;z)}{dt} = -zm^2(t;z) - 2q(t)m(t;z) + 2p(t)-z .
\]
We want to analyze this equation near $z=\infty$, so we take $z=\varphi( i+h)$. A calculation using geometric series shows that
\begin{equation}
\label{4.6}
\varphi (i+h) = \frac{1}{h} - \sum_{n=0}^{\infty} \left( \frac{i}{2}\right)^{n+1} h^n , \quad |h|<1 .
\end{equation}
Let's split off the dominant term and write $m(t;z)=i+g(t;z)$. We then have
\begin{equation}
\label{4.3}
g(t;\varphi(i+h)) = \sum_{n=1}^{\infty} f_n(t) h^n ,
\end{equation}
with $f_n(t)$ of course still defined as the expansion coefficient from \eqref{4.1}, but now for the function $m(t;\varphi(i+h))$.
The new function $g$ solves
\begin{equation}
\label{4.5}
\frac{dg}{dt} = -zg^2 -2(q+iz)g + 2(p-iq) .
\end{equation}
We now plug \eqref{4.3} into this and differentiate the expansion term by term. We obtain
\begin{align}
\label{4.4}
\sum_{n\ge 1}f'_n(t)h^n &= -z\sum_{m,n\ge 1} f_m(t)f_n(t) h^{m+n} \\
& \nonumber \quad\quad -2(q(t)+iz)\sum_{n\ge 1} f_n(t)h^n + 2(p(t)-iq(t)) ,
\end{align}
and here $z$ is short-hand for $z=\varphi(i+h)$.
Of course, this formal calculation has done little towards proving \eqref{4.4} rigorously. For starters, we don't even know at this point if the coefficients
$f_n$ are differentiable. However, \eqref{4.4} is correct, and $f_n\in C^{\infty}(\R)$. This may be established exactly as in \cite{HMcBR} by
temporarily working with the integrated version of \eqref{4.5}. Lemma \ref{L4.1} gives the uniform bounds $|f_n(t)|\le 1$, so convergence of the various series is never
an issue for (say) $|h|<1$. Then we compare coefficients in this rewriting of \eqref{4.4}, and we obtain integrated versions of
the identities that we will derive
below in our lazier presentation of the argument, and then an inductive argument will confirm that $f_n\in C^{\infty}$.
We will leave this technical matter at that and refer the reader to the proof of \cite[Theorem 4.1]{HMcBR} for further details.

So we may now compare the coefficients of $h^N$ to obtain information about the evolution of the $f_n(t)$ without further misgivings.
We of course use the expansion \eqref{4.6} of $z=\varphi(i+h)$ on the right-hand side. Start out perhaps as a quick consistency check by observing
that there are no terms proportional to $h^{-1}$ on the right-hand side even though this is how the expansion of $z$ starts. Next, comparing
coefficients of $h^0$, we find $0 = -2if_1 + 2(p-iq)$, so
\begin{equation}
\label{4.8}
f_1(t) = -q(t)-ip(t) .
\end{equation}
This is essentially \eqref{4.9}. To prove
this latter formula, we must not forget that in the statement of Theorem \ref{T4.1}, we worked with the $\tr W=0$ normalization rather than \eqref{4.10}.
To go back to a traceless $W$, we must act by $\alpha(x)=\int_0^x p(t)\, dt\in\mathcal A_0$ on our current $W$ to obtain
\begin{equation}
\label{4.11}
(\alpha\cdot W)(x) = R_{\alpha(x)}W(x)R^t_{\alpha(x)}+p(x) .
\end{equation}
In particular,
\[
(\alpha\cdot W)(0) = \begin{pmatrix} p(0) & q(0) \\ q(0) & -p(0) \end{pmatrix} ,
\]
so $a(0)=p(0)$, $b(0)=q(0)$, and \eqref{4.9} does follow.

It remains to discuss the bound $\|W(x)\|\le 1$. Since the set $\mathcal D(U)$ as well as the $\tr W=0$ normalization are invariant under shifts of $W$,
it suffices to consider a single point, which for convenience we take as $x=0$. We already proved \eqref{4.9}, and given this identity, the inequality
is immediate from Lemma \ref{L4.1}. This result also identifies the $F$ functions with $\|W(0)\|=1$ as $F=R_{\theta}\cdot\lambda$, and, as we already
discussed in the previous section, these correspond exactly to the potentials $W_{\beta}$ from \eqref{3.21}.
\end{proof}
After this highly promising start with \eqref{4.4}, it is tempting to try to squeeze more out this equation. Comparing the coefficients of $h^1$ produces
\begin{equation}
\label{f2}
f'_1 = -f^2_1 - 2qf_1-2if_2-f_1 .
\end{equation}
For $n\ge 2$, we obtain
\begin{align}
\label{4.12}
f'_n =& -\sum_{j=1}^n f_j f_{n+1-j} + \sum_{k=0}^{n-2} (i/2)^{k+1}\sum_{j=1}^{n-1-k} f_j f_{n-k-j}\\
\nonumber
& -2qf_n - 2if_{n+1} - \sum_{k=0}^{n-1} (i/2)^k f_{n-k} .
\end{align}
In fact, this formula also works for $n=1$ if we interpret the empty sum $\sum_{k=0}^{n-2} \ldots = 0$, as usual.

The sharp estimate on $\|W(x)\|$ of Theorem \ref{T4.1} came from Lemma \ref{L4.1} for $n=1$. We can use these bounds in the same
way for $n>1$ to estimate other combinations of the entries of $W$, which will now involve derivatives. Of course, as a quick glance at
\eqref{4.12} confirms, the concrete expressions get out of hand quickly.
So we will limit ourselves to just one small illustration.
\begin{Theorem}
\label{T4.2}
Let $W\in\mathcal D(U)$, $\tr W=0$; write $W=\left( \begin{smallmatrix} a & b \\ b & -a \end{smallmatrix}\right)$. Then
\[
\left( b^2(x)-a^2(x)+b(x)+b'(x)\right)^2 +\left( 2a(x)b(x)+a(x)+a'(x) \right)^2 \le 4
\]
for all $x\in\R$. Moreover, equality for a single $x_0\in\R$ implies that $a(x)\equiv 0$, $b(x)\equiv 1$.
\end{Theorem}
To \textit{prove }this, we proceed exactly as outlined, and since the argument is very similar to what we just did, we will only sketch it.
We can focus on $x=0$. We combine \eqref{4.8}, \eqref{f2} to produce a formula for $f_2$ in terms of $p,q,p',q'$ at $x=0$.
Then we use the bound $|f_2|\le 1$ from Lemma \ref{L4.1} and we transform back to $a,b$. This will give the stated bound.
Lemma \ref{L4.1} also says that $|f_2|=1$ is only possible
if $\nu=\delta_0$, which is the measure of the $F$ function $F(\lambda)=-1/\lambda$, and, as we saw earlier,
the corresponding $\tr W=0$ normalized Dirac potential is $a\equiv 0$, $b\equiv 1$.
\section{Proof of Theorem \ref{T1.2}, Dirac version}
In this section, we establish a version of Theorem \ref{T1.2} for Dirac operators. Let's give a precise formulation of what we will prove.
\begin{Theorem}
\label{T6.1}
Let $W\in\mathcal W [0,\infty)$, $\tr W=0$. Assume that $m_+(z;W)$ satisfies Hypothesis \ref{H1.1}.
Then $W(x)$ is real analytic. More precisely, the entries of $W(x)$ have holomorphic continuations $W(z)$ to a rounded strip
$S=\{ z\in\C: |z-x|<h \textrm{\rm { for some }} x>0\}$.

Moreover, we have $\|W^{(n)}(x)\| \le B_n$, $x\ge 0$, for all $n\ge 0$.
\end{Theorem}
The proof will also show that we have bounds $B_n$ that depend only on $r>0$ from Hypothesis \ref{H1.1},
not on $W$ itself.

We prepare for the proof with general bounds on the Taylor coefficients of certain holomorphic functions, which will be applied to $m(z)$ later.
Recall our notation $D_r=\{ z: |z|<r\}$.
\begin{Lemma}
\label{L6.1}
For any holomorphic function $f: D_r\to\C^+$, $f(0)=i$, we have
\[
\left| f^{(N)}(0) \right| \le (3/r)^N N! .
\]
\end{Lemma}
Of course, these bounds are not sharp asymptotically for a given $f$ when $N\to\infty$; even without using the fact that $f$ maps to $\C^+$,
just from the lower bound $r$ on the radius of convergence, we already obtain $|f^{(N)}(0)|\le C(q;f) q^NN!$ for any $q>1/r$.
However, and this is the point here, the bounds of Lemma \ref{L6.1} are uniform in $f$.
\begin{proof}
It suffices to prove this for $r=1$ since we can apply this version to $g(z)=f(rz)$ to obtain the general case. We use the conformal map
\[
\psi : D_1\to\C^+, \quad \psi(z) = i\, \frac{1+z}{1-z}
\]
to represent $f=F\circ \psi$, and here $F=f\circ\psi^{-1}$ now maps $F:\C^+\to\C^+$, $F(i)=i$. We thus know from Lemma \ref{L4.1} and its proof that
the Taylor coefficients of $F(w)-i =\sum_{n\ge 1} a_n(w-i)^n$ satisfy $|a_n|\le 1$. We have
\[
\psi(z)-i = \frac{2iz}{1-z} ,
\]
so
\begin{equation}
\label{6.1}
f(z)-i = \sum_{n\ge 1} a_n \left( \frac{2iz}{1-z} \right)^n ,\quad |z|<1/3 .
\end{equation}
We can find the Taylor coefficients of $(1-z)^{-n}$ explicitly, for example by starting with the geometric series and taking derivatives. We have
\[
\left( \frac{1}{1-z} \right)^n = \sum_{k\ge 0} \binom{k+n-1}{n-1} z^k , \quad |z|<1 .
\]
Plugging this back into \eqref{6.1} and collecting the terms contributing to $z^N$, we find that $f(z)-i =\sum_{N\ge 1} b_Nz^N$, with
\[
b_N = \sum_{n=1}^N (2i)^n a_n \binom{N-1}{n-1} .
\]
In particular,
\[
|b_N| \le \sum_{n=1}^N 2^n \binom{N-1}{n-1} = 2\sum_{n=0}^{N-1} \binom{N-1}{n} 2^n = 2\cdot 3^{N-1} ,
\]
by the binomial theorem in the last step. Since $b_N = f^{(N)}(0)/N!$, this implies the desired bound.
\end{proof}
\begin{proof}[Proof of Theorem \ref{T6.1}.]
As above, the argument proceeds by an analysis of the Riccati equation that is satisfied by
$m_+(t;z)\equiv m_+(z; t\cdot W)$, with $(t\cdot W)(x)=W(t+x)$. Here we will again work with the $W_{11}\equiv 0$
normalization from Proposition \ref{P2.1}. However, instead of $z=\varphi(i+h)$, we can now use the much simpler
map $z=-1/w$. We again consider
\[
g(t;w)=m_+(t;-1/w)-i .
\]
Then, as in the proof of Theorem \ref{T4.1} (compare \eqref{4.5}), we have
\begin{equation}
\label{6.5}
\frac{dg}{dt} = \frac{1}{w}g^2 -2qg+\frac{2i}{w}g + 2(p-iq) .
\end{equation}
At this point, for general $t>0$,
we can only be really certain about this for $w\in\C^+$ so that $g$ is defined in terms of the original $m$ function, with no analytic continuation necessary.
However, and this is a key point, the shifted $m$ functions do satisfy Hypothesis \ref{H1.1} as well, for the same $r>0$.

To see this, recall that $m_+(t;z) = T(t;z)\cdot m_+(0;z)$, with
$T$ denoting the (full, with $z$ restored) transfer matrix. In other words, $T$ is the matrix solution of \eqref{Dirac} with the initial value $T(0;z)=1$.
By assumption, $m_+(0;z)$ has a holomorphic continuation
\[
M: \{ z\in\C_{\infty}: |z|>1/r \}\to\C^+ .
\]
Thus $T(t;z)\cdot M(z)$, at this point viewed as a function taking values in $\C_{\infty}$, is a holomorphic continuation
of $m_+(t;z)$ to the punctured neighborhood $\{z\in\C : |z|>1/r\}$ of $\infty$. For fixed $z\in\C^-\cup\R$, the transfer matrix $T(t;z)$, acting as a
linear fractional transformation $w\mapsto T(t;z)\cdot w$, maps $\C^+$ back to $\C^+$. This follows by combining \cite[Theorem 1.2]{Rembook}, or rather
the version of this result for $z\in\C^-\cup\R$ and Dirac equations (which has the same proof), with \cite[Lemma 3.9]{Rembook}.
So $T(t;z)\cdot M(z)\in\C^+$ for $z\in\C^-\cup\R$, $|z|>1/r$. This is also
trivially true for $z\in\C^+$ because then this function is simply the original $m$ function $m_+(t;z)$.

So we now have a holomorphic continuation of $m_+(t;z)$ to $|z|>1/r$ that maps to $\C^+$. In this situation, the isolated singularity at $z=\infty$
is removable. We have verified that $m_+(t;z)$ satisfies Hypothesis \ref{H1.1} for all $t\ge 0$, with a single uniform $r>0$.

In particular, for general $t\ge 0$, we may expand
\begin{equation}
\label{6.2}
g(t;w) = \sum_{n\ge 1} g_n(t) w^n , \quad |w|<r .
\end{equation}
Moreover, Lemma \ref{L6.1} applies to $w\mapsto g(t;w)+i$ for any fixed $t$ and shows that
\begin{equation}
\label{6.3}
|g_n(t)|\le (3/r)^n .
\end{equation}

As in the proof of Theorem \ref{T4.1}, we now plug \eqref{6.2} into \eqref{6.5}, and our plan is to compare coefficients to extract more information. We obtain
\begin{align}
\label{6.4}
\sum_{n\ge 1} g'_n(t)w^n =& \sum_{m,n\ge 1} g_m(t)g_n(t)w^{m+n-1} -2q(t)\sum_{n\ge 1} g_n(t)w^n\\ \nonumber
& +2i \sum_{n\ge 1}g_n(t)w^{n-1} + 2(p(t)-iq(t)) .
\end{align}
Of course, this purely formal calculation is open to all the same criticisms as before. To establish \eqref{6.4} rigorously, we really have to temporarily work with
the integrated version of \eqref{6.5}.
This will show that \eqref{6.4} is valid and $g_n\in C^{\infty}([0,\infty))$. We again refer the reader to \cite{HMcBR} for the details of these arguments.

By \eqref{6.3}, all series converge on $w\in D_{r/3}$ at least, for all $t\ge 0$. So we may now again compare coefficients. Starting with $w^0$, we find
that $0= 2ig_1 + 2(p-iq)$ or
\begin{equation}
\label{6.8}
g_1(t) = q(t)+ip(t) .
\end{equation}
Then, for general $n\ge 1$, we have the following simpler version of \eqref{4.12}:
\begin{equation}
\label{6.18}
g'_n = \sum_{j=1}^n g_jg_{n+1-j} -2qg_n +2ig_{n+1} .
\end{equation}
This will lead to the following bounds.
\begin{Lemma}
\label{L4.2}
There are $C,M>0$ such that
\[
\left| \frac{d^Ng_1(t)}{dt^N} \right| \le CM^N N!
\]
for all $N\ge 0$, $t\ge 0$.
\end{Lemma}
Let us postpone the slightly convoluted proof of Lemma \ref{L4.2}. We first explain how we can use these bounds to finish the proof of Theorem \ref{T6.1}.
Using Taylor's theorem with the Lagrange formula for the remainder, we deduce from Lemma \ref{L4.2} that $g_1(t)$ is represented by its power series about any $t_0>0$
for $t\in (t_0-1/M, t_0+1/M)$, $t\ge 0$. This power series then also provides a holomorphic continuation of $g_1$ to the disk
$\{ z\in\C : |z-t_0|<1/M\}$. When the disks overlap, these functions are holomorphic continuations of each other, and since the set
$S$ from Theorem \ref{T6.1} is simply connected, we obtain a holomorphic continuation of $g_1(t) =q(t)+ip(t)$ to all of $S$.

Since $p(t),q(t)$ are real valued, we have $q(t) = (g_1(t)+g^{\#}_1(t))/2$, $p(t) = (g_1(t)-g^{\#}_1(t))/(2i)$ for $t\ge 0$,
and the functions on the right-hand sides are also holomorphic on $S$. Here, we employ the usual notation $f^{\#}(z)=\overline{f(\overline{z})}$.
We have shown that $p,q$ themselves have holomorphic extensions to $S$. This is still not exactly what Theorem \ref{T6.1} claims because
of the discrepancy in normalizations. We must return to \eqref{4.11} to extract the entries $a,b$ of the original $\tr W=0$ version
as the matrix elements of $\alpha\cdot W$. A calculation shows that for example
\[
a(x) = p(x)\cos 2\alpha(x) - q(x) \sin 2\alpha(x) ,
\]
with $\alpha$ still given by $\alpha(x) = \int_0^x p(t)\, dt$. It follows that $a$ also has a holomorphic continuation to $S$, as desired.
Here, we can obtain the continuation of $\alpha$ by using the same formula for $\alpha(z)$ and interpreting the integral as a complex line integral.

Finally, Lemma \ref{L4.2} also provides the asserted bounds on $|q^{(n)}(t)+ip^{(n)}(t)|=|g^{(n)}_1(t)|$, and then on $a^{(n)}$, $b^{(n)}$ as well.

To prove that these bounds can be chosen to depend on $r>0$ from Hypothesis \ref{H1.1} only, as claimed in the comment following the statement of Theorem \ref{T6.1},
the reader would have to check that the upcoming proof of Lemma \ref{L4.2} produces constants $C,M$ that only depend on $r>0$.
We will not address this issue explicitly, but it is easy
to extract this extra information from the proof we are about to give.
\end{proof}
\begin{proof}[Proof of Lemma \ref{L4.2}.]
We differentiate \eqref{6.18} $N$ times, using the general product rule $(uv)^{(N)}=\sum_{d=0}^N \binom{N}{d}u^{(d)}v^{(N-d)}$.
It will be convenient to state the result of this calculation in terms of $h_n(N;t)=g^{(N)}_n(t)/N!$. We obtain
\begin{align}
\label{6.7}
(N+1)h_n(N+1;t) =& \sum_{d=0}^N\sum_{j=1}^n h_j(d;t)h_{n+1-j}(N-d;t)  \\ \nonumber
& -2\sum_{d=0}^N \frac{q^{(d)}(t)}{d!} h_n(N-d;t) + 2i h_{n+1}(N;t) .
\end{align}
We only want bounds, so we may modify these equations. We break the procedure into two easy small steps.

First define constants $A_n(N)>0$ for $n\ge 1$, $N\ge 0$ recursively by $A_n(0)=(3/r)^n$ and then,
for $N\ge 0$,
\begin{align*}
(N+1)A_n(N+1) = & \sum_{d=0}^N\sum_{j=1}^n A_j(d)A_{n+1-j}(N-d) \\
& +2\sum_{d=0}^N A_1(d) A_n(N-d) + 2A_{n+1}(N) .
\end{align*}
Then $|h_n(N;t)|\le A_n(N)$ for all $n\ge 1$, $N\ge 0$, $t\ge 0$.

We can prove this claim by a straightforward induction on $N$. The case $N=0$ is covered \eqref{6.3}. In the inductive step, we estimate
the right-hand side of \eqref{6.7} in the obvious way. To bound the first term in the second line, we use \eqref{6.8}, which shows that
$|q^{(d)}(t)/d!|\le |h_1(d;t)|$. So if we assume the induction hypothesis, this observation lets us estimate $|q^{(d)}(t)/d!|\le A_1(d)$.

We can further improve this recursion by defining $B_n(0)=(3/r)^n$ and then recursively
\begin{equation}
\label{6.9}
(N+1)B_n(N+1) =  3\sum_{d=0}^N\sum_{j=1}^n B_j(d)B_{n+1-j}(N-d)  + 2nB_{n+1}(N) .
\end{equation}
Then $B_n(N)\ge A_n(N)$ for all $n,N$, as another induction on $N$ will confirm.

Now consider the following Cauchy problem for a function $Y(w,t)$ of two variables near $(w,t)=(0,0)$:
\begin{equation}
\label{4.84}
\frac{\partial Y}{\partial t} = 3Y^2 + 2\frac{\partial Y}{\partial w} ,\quad\quad Y(w,0) = \frac{3}{r-3w} .
\end{equation}
This has a real analytic solution in a neighborhood of $(w,t)=(0,0)$, as we can confirm by solving \eqref{4.84} explicitly, or by (unnecessarily, but more conveniently)
referring to the Cauchy-Kovalevskaya theorem \cite[Theorem 4.6.2]{Evans}. So there are coefficients $C_n(N)$ and a $\rho>0$ such that
\begin{equation}
\label{4.86}
Y(w,t) = \sum_{n\ge 1, N\ge 0} C_n(N) w^{n-1}t^N , \quad |w|, |t| < \rho .
\end{equation}
Since $Y(w,0)=3/(r-3w)$, we have $C_n(0)=(3/r)^n$. Moreover, by plugging \eqref{4.86} back into \eqref{4.84} and comparing the coefficients
of $w^{n-1}t^N$, we see that the $C_n(N)$ also satisfy the
recursion \eqref{6.9}. Hence $C_n(N)=B_n(N)$ for all $n\ge 1$, $N\ge 0$.

We can now specialize to $n=1$. The $B_1(N)$ have been recognized as the coefficients of a power series with positive radius of convergence,
hence $B_1(N)\le CM^N$ for suitable $C,M>0$. This gives the bounds on $g^{(N)}_1(t)$ that were stated in Lemma \ref{L4.2}.
\end{proof}
\section{General canonical systems}
It is natural to ask how exactly the Dirac operators $\mathcal D(U)$ sit inside the more general canonical systems $\mathcal R(U)$, and one simple
but satisfactory answer is provided by the natural group action of $\PSL (2,\R)=\SL(2,\R)/\{\pm 1\}$ on $\mathcal C(\R)$. We have used this action extensively
already for the more specialized maps $R_{\theta}$, and we can define it in the same way in general. Recall that any $A\in\PSL(2,\R)$ acts on $\C_{\infty}$ as a linear
fractional transformation, as spelled out in \eqref{2.57},
and then also on canonical systems $H\in\mathcal C(\R)$ by acting pointwise in this way on the $m$ functions:
\begin{equation}
\label{5.1}
\pm m_{\pm}(z; A\cdot H) = A\cdot (\pm m_{\pm}(z; H)) .
\end{equation}
Then $A\cdot H$ is unitarily equivalent to $H$, and the property of being reflectionless on a set is preserved \cite[Theorems 7.2, 7.9(a)]{Rembook}.
In particular, the spaces $\mathcal R_0(U)$, $\mathcal R(U)$ are invariant under the action.
Recall also that these spaces (for any Borel set $U\subseteq\R$) always contain the trivial canonical systems $H\equiv P_{\alpha}$. Their $m$ functions
$m_{\pm}(z; P_{\alpha}) \equiv \mp\tan\alpha$ are constant, with value in $\R_{\infty}$. Recall the abbreviation $\mathcal Z$ from \eqref{1.93}
that we use for the collection of these systems.
Clearly, $\mathcal Z$ is also invariant under the action of $\PSL(2,\R)$.

As we saw much earlier, the subgroup $\operatorname{PSO}(2)$
preserves $\mathcal D_0(U)$, $\mathcal D(U)$. So this part of $\PSL(2,\R)$ is useless for our current purposes, and we focus on the dilation/translation subgroup
\[
G = \left\{ \begin{pmatrix} c & a/c \\ 0 & 1/c \end{pmatrix} : a\in\R, c>0 \right\}
\]
instead. Then $m_+(z; g\cdot H) = c^2m_+(z;H)+a$. If we also recall that $m_+$ is holomorphic at $z=\infty$ if $H\in\mathcal R(U)$ and
$H\in\mathcal D(U)$ if and only if $m_+(\infty)=i$, by Theorem \ref{T3.2}, then the following result is now obvious.
\begin{Theorem}
\label{T5.1}
Let $H\in\mathcal R(U)$. Then either $H\in\mathcal Z$, or the orbit $\{g\cdot H: g\in G\}$ contains a unique
$H_1=g\cdot H\in\mathcal D(U)$; in this latter case, $g\in G$ is also unique.

If, in addition, $H\in\mathcal R_0(U)\setminus\mathcal Z$, then $g\cdot H\in\mathcal D_0(U)$ also.
\end{Theorem}
Please see also \cite{FR} for much more on this general theme of the $\PSL(2,\R)$ action on spaces of reflectionless canonical systems.
\begin{Corollary}
\label{C5.1}
Let $H\in\mathcal R(U)$. Then either $H\in\mathcal Z$, or $\det H(x)>0$ for all $x\in\R$. In this second case, we may normalize
$H$ by demanding that $\det H(x)=1$. The entries of this version of $H(x)$ are real analytic functions of $x\in\R$.
\end{Corollary}
\begin{proof}
Recall that the group action \eqref{5.1} has the following effect on the coefficient functions \cite[Theorem 3.20]{Rembook}: $(A\cdot H)(x) = A^{-1t}H(x)A^{-1}$.

Now assume that we are not dealing with one of the degenerate systems $H\equiv P_{\alpha}$. By Theorem \ref{T5.1}, $H(x)=B^t H_d(x) B$
for some $B\in\SL(2,\R)$ and some $H_d\in\mathcal D(U)$. Since $\det H_d=1$ was in fact our standard normalization for such $H_d$,
it is clear that $H$ can be given this same property, and the representative $B^tH_dB$ itself is already normalized in this way.

Next, recall that $H_d(x)=T^t(x)T(x)$, with $T$ denoting the solution of $JT'+WT=0$, $T(0)=1$. We know that $W$ is real analytic (in some
normalizations at least, including the $\tr W=0$ version), and holomorphic
differential equations have holomorphic solutions \cite[Theorem 1.8.1]{CL}, hence $T$, $H_d$, and $H$ are real analytic as well.
\end{proof}
Having clarified this, we can now also state the bounds of Theorem \ref{T4.1} in terms of $H$ directly, for general $H\in\mathcal R(U)$.
\begin{Theorem}
\label{T5.2}
Let $H\in\mathcal R(U)\setminus\mathcal Z$, and normalize $H$ as usual by requiring $\det H(x)=1$. Then
\begin{equation}
\label{5.11}
\| H^{-1/2}(x)H'(x)H^{-1/2}(x)\|\le 2
\end{equation}
for all $x\in\R$. Moreover, equality for a single $x_0\in\R$ implies that $H\in\mathcal R_0(U)$.
\end{Theorem}
\begin{proof}
Since $\mathcal R_0(U), \mathcal R(U),\mathcal Z$ are invariant under shifts, it suffices to discuss the case $x=x_0=0$. As above, we have
$H(x)=A^tH_d(x)A$ for some $A\in\SL (2,\R)$, with $H_d=T^tT\in\mathcal D(U)$. We compute $H'_d(0)=T'^t(0)+T'(0)$.
To evaluate this, write $W=\left( \begin{smallmatrix} a & b \\ b & -a \end{smallmatrix}\right)$. Since $T'(0)=JW(0)$, we obtain
\begin{equation}
\label{5.12}
H'_d(0) = 2\begin{pmatrix} -b(0) & a(0) \\ a(0) & b(0) \end{pmatrix} .
\end{equation}

Consider now the matrix from \eqref{5.11}; call it $K(x)$.
Since $\det H(x)\equiv 1$, we have $\tr H'H^{-1}=0$, so $\tr K(x)=0$ as well. It follows that the eigenvalues of $K$ are $\pm (-\det K)^{1/2}$.
Since $K$ is symmetric, this implies that $\|K\| = |\det K|^{1/2}$. Moreover, $\det K = \det H' = \det H'_d$,
and at $x=0$ this equals $-4(a^2(0)+b^2(0))$, by \eqref{5.12}.
By Theorem \ref{T4.1}, $a^2+b^2\le 1$, with equality precisely when $H_d\in\mathcal D_0(U)$, and this holds if and only if
$H\in\mathcal R_0(U)$.
\end{proof}
Note that it would not do to try to run this argument in the original setting of Theorem \ref{T4.1} because a shift $t\cdot H$ of an $H\in\mathcal D(U)$
will typically no longer satisfy $(t\cdot H)(0)=H(t)=1$, and thus we have left $\mathcal D(U)$. We really need the larger space $\mathcal R(U)$.
The issues discussed in Theorem \ref{T2.3} are again lurking behind the scenes.

We are now also finally ready for the
\begin{proof}[Proof of Theorem \ref{T1.2}.]
In fact, exactly the same ideas let us go from Theorem \ref{T6.1} to its more general version Theorem \ref{T1.2}. It is still true that if $H\in\mathcal C[0,\infty)$
satisfies Hypothesis \ref{H1.1}, then $H\in\mathcal D$ if and only if $m_+(\infty)=i$. The argument is literally the same as in the proof of Theorem \ref{T3.2}.

Hence, as above, if an $H\in\mathcal C[0,\infty)$ satisfying Hypothesis \ref{H1.1} is given and $H\not\equiv P_{\alpha}$, then $g\cdot H\in\mathcal D$ for a unique
$g\in G$. We simply need to act by that $g\in G$ that corrects the value at infinity, that is, we need $g\cdot m_+(\infty; H)=i$.

Furthermore, since $g$ maps $\C^+$ back to itself when acting as a linear fractional transformation, $g\cdot H$ still satisfies Hypothesis \ref{H1.1}.

The bottom line is that an $H$ satisfying Hypothesis \ref{H1.1}, $H\not\equiv P_{\alpha}$,
is again of the form $H(x)=A^tH_d(x)A$, $H_d=T^tT$, with $T$ solving $JT'+WT=0$, and here $W$ is a Dirac potential of the type analyzed
in Theorem \ref{T6.1}. We conclude that $H_d$ and $H$ are indeed real analytic, with holomorphic continuations to $S$.

The bound on $H^{-1/2}H'H^{-1/2}$ from Theorem \ref{T1.2} is established in the same way as the analogous claim of Theorem \ref{T5.2}.
Of course, since we do not have a sharp bound here, there are no additional claims that can be made when the bound is attained.
\end{proof}
We conclude with a few remarks on other sets $U\subseteq\R$. Sets of the type $U=(-2,2)$ and $U=(0,\infty)$ were already considered
in \cite{HMcBR}, but only for Jacobi and Schr{\"o}dinger operators, respectively, not for general canonical systems. Of course,
we may expect an analog of Theorem \ref{T5.1} that lets one move from a general canonical system $H\in\mathcal R(U)$ to one of these more specialized operators
by acting by a suitable group element $A\in\PSL(2,\R)$, and, assuming this, \cite{HMcBR} could then be upgraded effortlessly to the more general setting.
This is essentially correct (and see again \cite{FR} also), but only
if a smaller (than $\mathcal R(U)$) class of canonical systems is considered. A natural choice, in the second case $U=(0,\infty)$, say, would be to
work only with canonical systems that are bounded below. It is in fact clear that something of this sort is needed because whether or not a canonical system
can be rewritten as a Schr{\"o}dinger equation is again decided by the large $z$ asymptotics of the $m$ functions \cite{Lev,LevSar,Mar,RemdB}, and
a spectral measure that is unbounded below can now certainly change the behavior of $m(z)$ near $z=\infty$ dramatically. This is also the approach
already taken in \cite{HMcBR}. Nothing like this was needed here, for Dirac operators, because now \textit{any }measure on $U^c=[-1,1]$ will only have a small effect
on the large $z$ asymptotics of $m(z)$.

Finally, one can consider completely general open sets $U\subseteq\R$. Then $\Omega=\C^+\cup U\cup\C^-$ will usually not be simply connected.
We can work with the universal cover $\varphi: \C^+\to\Omega$ as a substitute for the conformal map that is no longer available.
Now we obtain as $F$ functions $F(\lambda;H)=M(\varphi(\lambda))$, $H\in\mathcal R(U)$, exactly the \textit{automorphic }Herglotz functions,
that is, the functions invariant under the (Fuchsian) group $G\le\PSL(2,\R)$ of covering transformations. Though explicit calculations such
as the ones from Sections 4, 5 may become challenging, even in relatively simple cases,
it appears that method still has potential. This will be the subject of continuing research.


\begin{thebibliography}{100}
\bibitem{Arch} K.\ Archaya, An alternate proof of the de~Branges theorem on canonical systems,
\textit{ISRN Math.\ Anal.\ }(2014), 7 pp.
\bibitem{BLY} R.\ Bessonov, M.\ Lukic, and P.\ Yuditskii, A theory of reflectionless canonical systems, I.\ Arov gauge and right limits,
\textit{Int.\ Eq.\ Op.\ Theory }\textbf{94 }(2022), 30 pp.
\bibitem{CL} E.A. Coddington and N.\ Levinson, Theory of ordinary differential equations, McGraw-Hill, New York, 1955.
\bibitem{Evans} L.C.\ Evans, Partial differential equations, Graduate studies in mathematics 19, American Mathematical Society, Rhode Island, 2002.
\bibitem{EHS} W.N.\ Everitt, D.B.\ Hinton, and J.K.\ Shaw, The asymptotic form of the Titchmarsh-Weyl coefficient for Dirac systems,
\textit{J.\ London Math.\ Soc.\ }\textbf{27 }(1983), 465--476.
\bibitem{FR} M.\ Forester and C.\ Remling, Topological properties of reflectionless canonical systems,
preprint, \texttt{https://arxiv.org/pdf/2409.04862.}
\bibitem{GasLev} M.G.\ Gasymov and B.M. Levitan, The inverse problem for the Dirac system (in Russian),
\textit{Dokl.\ Akad.\ Nauk SSSR }\textbf{165 }(1966), 967--970.
\bibitem{GesHol1}F.\ Gesztesy and H.\ Holden, Soliton equations and their algebro-geometric
solutions, vol.\ I, (1+1)-dimensional continuous models, Cambridge Studies in
Advanced Mathematics, 79, Cambridge University Press, Cambridge, 2003.
\bibitem{GesHol2}F.\ Gesztesy, H.\ Holden, J.\ Michor, and G.\ Teschl, Soliton equations and
their algebro-geometric solutions, vol.\ II, (1+1)-dimensional discrete models,
Cambridge Studies in Advanced Mathematics, 114, Cambridge University
Press, Cambridge, 2008.
\bibitem{HarAsh}T.\ Harutyunyan and Y.\ Ashrafyan, Spectral theory of Dirac operators,
preprint (book), \texttt{https://arxiv.org/abs/2403.02761.}
\bibitem{HMcBR}I.\ Hur, M.\ McBride, and C.\ Remling, The Marchenko representation of reflectionless Jacobi and Schrödinger operators,
\textit{Trans.\ Amer.\ Math.\ Soc.\ }\textbf{368 }(2016), 1251--1270.
\bibitem{Lev}B.M.\ Levitan, Inverse Sturm-Liouville problems, VNU Science Press, Utrecht, 1987.
\bibitem{LevSar}B.M.\ Levitan and I.S.\ Sargsjan, Sturm-Liouville and Dirac operators,
Mathematics and its Applications 59, Springer Science + Business Media, Dordrecht, 1991.
\bibitem{Mar}V.A. Marchenko, Sturm-Liouville operators and applications, Birkh{\"a}user, Basel, 1986.
\bibitem{Mar} V.A.\ Marchenko, The Cauchy problem for the KdV equation with nondecreasing initial data,
in \textit{What is integrability?, }273–318, Springer, Berlin, 1991.
\bibitem{Mum}D.\ Mumford, Tata Lectures on Theta 2, Birkh{\"a}user-Verlag, Basel, 1984.
\bibitem{MykSush1} Y.V.\ Mykytyuk and N.\ Sushchyk, The strip of analyticity of reflectionless potentials,
\textit{Mat.\ Stud.\ }\textbf{57 }(2022), 186--191.
\bibitem{MykSush2} Y.V.\ Mykytyuk and N.\ Sushchyk, Reflectionless Schr{\"o}dinger operators and Marchenko parametrization,
\textit{Mat.\ Stud.\ }\textbf{61 }(2024), 79--83.
\bibitem{Phe} R.R.\ Phelps, Lectures on Choquet’s theorem, 2nd edition, Lecture Notes in
Mathematics, 1757, Springer-Verlag, Berlin, 2001
\bibitem{RemdB}C.\ Remling, Schr{\"o}dinger operators and de~Branges spaces,
\textit{J.\ Funct.\ Anal.\ }{\textbf 196 }(2002), 323--394.
\bibitem{RemAnn} C.\ Remling, The absolutely continuous spectrum of Jacobi matrices,
\textit{Annals of Math.\ }\textbf{174 }(2011), 125--171.
\bibitem{Rembook}C.\ Remling, Spectral theory of canonical systems, de~Gruyter Studies in Mathematics 70,
Berlin/Boston, 2018.
\bibitem{RemToda}C.\ Remling, Toda maps, cocycles, and canonical systems,
\textit{J.\ Spectral Theory }\textbf{9 }(2019), 1327--1365.
\bibitem{Simr1} B.\ Simon, Spectral analysis of rank one perturbations and applications,
in \textit{Proc.\ Mathematical Quantum Theory, II: Schr{\"o}dinger Operators }(eds.\ J.\ Feldman, R.\ Froese, and L.\ Rosen),
CRM Proc.\ Lecture Notes 8 (1995), 109--149.
\bibitem{Teschl}G.\ Teschl, Jacobi operators and completely integrable nonlinear lattices,
Mathematical Monographs and Surveys 72, American Mathematical Society,
Providence, 2000.
\bibitem{ZengPhD} J.\ Zeng, Spectral theory of Dirac operators with measures,
PhD thesis, University of Oklahoma, Norman, 2023.
\end{thebibliography}
\end{document}